\documentclass[12pt]{article}

\usepackage{amssymb,amsmath,amsthm}
\usepackage[shortlabels]{enumitem}
\usepackage{xcolor}

\newcommand{\Z}{\mathbb{Z}}

\newcommand{\Q}{\mathbb{Q}}
\newcommand{\F}{\mathbb{F}}

\renewcommand{\SS}{\mathcal{S}}

\newcommand{\OO}{\mathcal{O}}
\newcommand{\U}{\mathcal{U}}
\newcommand{\M}{\mathcal{M}}

\newcommand{\xbar}{\overline{x}}
\newcommand{\ybar}{\overline{y}}
\newcommand{\zbar}{\overline{z}}

\newcommand{\Pbar}{\overline{P}}

\newcommand{\psibar}{\overline{\psi}}

\newcommand{\Gbar}{\overline{G}}
\newcommand{\ra}{\rightarrow}

\newcommand{\tst}{\textstyle}

\newcommand{\sigmat}{\widetilde{\sigma}}
\newcommand{\taut}{\widetilde{\tau}}

\newcommand{\At}{\widetilde{A}}
\newcommand{\Gt}{\widetilde{G}}

\newcommand{\id}{\text{id}}

\DeclareMathOperator{\Gal}{Gal}
\DeclareMathOperator{\rank}{rank}

\DeclareMathOperator{\Aut}{Aut}

\newcommand{\Lrightarrow}{\hbox to1cm{\rightarrowfill}}
\newcommand{\Ldownarrow}{\bigg\downarrow}

\newtheorem{theorem}{Theorem}
\newtheorem{lemma}[theorem]{Lemma}
\newtheorem{prop}[theorem]{Proposition}
\newtheorem{cor}[theorem]{Corollary}

\theoremstyle{definition}
\newtheorem{remark}[theorem]{Remark}

\numberwithin{equation}{section}
\numberwithin{theorem}{section}

\title{A converse to the Hasse-Arf theorem}
\author{G. Griffith Elder \\
Department of Mathematics \\
University of Nebraska Omaha \\
Omaha, NE 68182 \\
USA \\[.2cm]
{\tt elder@unomaha.edu}
\and
Kevin Keating \\
Department of Mathematics \\
University of Florida \\
Gainesville, FL 32611 \\
USA \\[.2cm]
{\tt keating@ufl.edu}}
\begin{document}

\maketitle

\begin{abstract}
Let $L/K$ be a finite Galois extension of local fields.
The Hasse-Arf theorem says that if $\Gal(L/K)$ is
abelian then the upper ramification breaks of $L/K$ must
be integers.  We prove the following converse to the
Hasse-Arf theorem: Let $G$ be a nonabelian group which
is isomorphic to the Galois group of some totally
ramified extension $E/F$ of local fields with residue
characteristic $p>2$.  Then there is a totally ramified
extension of local fields $L/K$ with residue
characteristic $p$ such that $\Gal(L/K)\cong G$ and
$L/K$ has at least one nonintegral upper ramification
break.
\end{abstract}

\section{Introduction}

Let $K$ be a local field and let $L/K$ be a finite
Galois extension.  Associated to $L/K$ are rational
numbers $u_1\le u_2\le\cdots\le u_n$ known as the upper
ramification breaks of $L/K$.  The upper ramification
breaks provide arithmetic information about the
extension $L/K$.  For instance, $L/K$ is a nontrivial
unramified extension if and only if $-1$ is the only
upper ramification break of $L/K$, and $L/K$ is at most
tamely ramified if and only if the set of upper
ramification breaks of $L/K$ is contained in $\{-1,0\}$.
It is a classical problem to determine the possibilities
for sequences of upper breaks.  The Hasse-Arf theorem
\cite{hasse,arf} says that if $G=\Gal(L/K)$ is abelian
then every upper break of $L/K$ is an integer.  The
Hasse-Arf theorem plays an important role in several
areas of number theory.  For instance, it is used in the
construction of the Artin representation \cite{artin}
and in Lubin's proof of the local Kronecker-Weber
theorem \cite{lcw}.

     The purpose of this paper is to prove a converse to
the Hasse-Arf theorem.  A full converse to the Hasse-Arf
theorem would state that if $G$ is a finite nonabelian
group then there exists a $G$-extension of local fields
which has a nonintegral upper ramification break.  In
fact the converse to Hasse-Arf does not hold in such
generality.  For instance, if $G$ is a nonabelian simple
group and $L/K$ is a $G$-extension then $K$ has infinite
residue field and $L/K$ is unramified, so the only upper
ramification break of $L/K$ is $-1$.  In
addition, if $G$ is a nonabelian group of order prime
to $p$ then every $G$-extension $L/K$ of local fields
with residue characteristic $p$ is at most tamely
ramified, and hence has upper ramification breaks
contained in $\{-1,0\}$.  To rule out examples like
these we restrict our attention to totally ramified
extensions.  Furthermore, to avoid vacuous cases we only
consider those nonabelian groups $G$ which can actually
occur as the Galois group of a totally ramified
extension of local fields with residue characteristic
$p$.  In Theorem~\ref{chat} we prove that if $G$ is a
such a group then for every local field $K$ of
characteristic $p$ there exists a totally ramified
$G$-extension $L/K$ which has a nonintegral upper
ramification break.  It then follows from a theorem of
Deligne that there exists a local field $F$ of
characteristic 0 and a totally ramified $G$-extension
$E/F$ which has a nonintegral upper ramification break.

     In Section~\ref{ram} we outline higher ramification
theory for Galois extensions of local fields.  Our
approach to proving the converse to Hasse-Arf is based
on constructing Galois extensions of local fields in
characteristic $p$ which have nonintegral upper
ramification breaks.  A benefit of working in
characteristic $p$ is that one can solve embedding
problems for $p$-extensions, as explained in
Section~\ref{emb}.  This means that it's enough to
construct extensions whose Galois groups are minimal in
a certain sense.  Therefore in Section~\ref{min} we
classify the minimal $p$-groups.  In Section~\ref{pext}
we use the results of Sections~\ref{emb} and \ref{min}
to prove the converse to the Hasse-Arf theorem for
totally ramified $p$-extensions.  In Section~\ref{Gext}
we extend the proof to cover arbitrary totally ramified
extensions.

     Throughout the paper we let $K$ be a field which is
complete with respect to a discrete valuation, with
perfect residue field of characteristic $p>2$.  Let
$K^{sep}$ be a separable closure of $K$, and for each
finite subextension $L/K$ of $K^{sep}/K$ let $v_L$ be
the valuation on $K^{sep}$ normalized so that
$v_L(L^{\times})=\Z$.  Let $\OO_L$ denote the ring of
integers of $L$, let $\M_L$ denote the maximal ideal of
$\OO_L$, and let $\pi_L$ be a uniformizer for $L$.

\section{Ramification in extensions of local fields}
\label{ram}

Let $L/K$ be a finite totally ramified Galois extension.
In this section we define the lower and upper
ramification breaks of $L/K$, and the ramification
subgroups of $\Gal(L/K)$.  For more information on these
topics see Chapter~IV of \cite{cl}.

     Let $L/K$ be a totally ramified Galois extension of
degree $mp^n$, with $p\nmid m$.  Set $G=\Gal(L/K)$.  For
$\sigma\in G$ with $\sigma\not=\id_L$ define the
ramification number of $\sigma$ to be
$i(\sigma)=v_L(\sigma(\pi_L)-\pi_L)-1$; also define
$i(\id_L)=\infty$.  (Beware that $i(\sigma)$ is related
to $i_G(\sigma)$ as defined in \cite[IV]{cl} by
$i_G(\sigma)=i(\sigma)+1$.)  One easily sees that if
$0<i(\sigma)<\infty$ then for all $\tau\in G$ we have
\begin{equation} \label{bigger}
i(\sigma^p)>i(\sigma),\;\;i([\sigma,\tau])>i(\tau).
\end{equation}
For real $x\ge0$ set
$G_x=\{\sigma\in G:i(\sigma)\ge x\}$.  Then $G_x$ is a
normal subgroup of $G$, known as the $x$th lower
ramification subgroup of $G$.  Say $b\ge0$ is a lower
ramification break of $L/K$ if
$G_{b+\epsilon}\lneqq G_b$ for all $\epsilon>0$.  Thus
$b$ is a lower break of $L/K$ if and only if
$b=i(\sigma)$ for some $\sigma\in G$ with
$\sigma\not=\id_G$.  It follows that every lower break
of $L/K$ is a nonnegative integer.  Furthermore, a
nonnegative integer $b$ is a lower ramification break if
and only if $G_{b+1}\lneqq G_b$.

     We have $i(\sigma)=0$ if and only if $|\sigma|$ is
not a power of $p$.  Hence $b_0=0$ is a lower
ramification break of $L/K$ if and only if $m>1$.  If
$b$ is a positive lower break of $L/K$ then
$|G_b:G_{b+1}|=p^d$ for some $d\ge1$.  In this case we
say that $b$ is a lower break with multiplicity $d$.
The positive lower ramification breaks of $L/K$, counted
with multiplicities, form a multiset with cardinality
$n$.  We denote the positive lower breaks of $L/K$ by
$b_1\le b_2\le\cdots\le b_n$.

     Let $M/K$ be a subextension of $L/K$ and set
$H=\Gal(L/M)$.  It follows from the definitions that
$H_x=H\cap G_x$ for all $x\ge0$.  Therefore the multiset
of lower ramification breaks of $L/M$ is contained in
the multiset of lower ramification breaks of $L/K$.  In
other words, the ramification groups $G_x$ and the lower
ramification breaks are compatible with passage to
subgroups of Galois groups.

     There is a different numbering system for the
ramification groups of $G=\Gal(L/K)$ which is compatible
with passage to quotients $G/H$ of $G$ by a normal
subgroup $H$.  The upper ramification breaks of $L/K$
are defined in terms of the lower breaks as follows:
First, $u_0=0$ is an upper break of $L/K$ if and only if
$b_0=0$ is a lower break.  The positive upper breaks
$u_1\le u_2\le\cdots\le u_n$ of $L/K$ are then defined
recursively by $u_1=b_1/m$ and
$u_{i+1}-u_i=(b_{i+1}-b_i)/mp^i$ for $1\le i\le n-1$.
We may view $u_i$ as the upper ramification break of
$L/K$ which corresponds to $b_i$.  The upper
ramification breaks of $L/K$, counted with
multiplicities, form a multiset, which we denote by
$\U_{L/K}$.  Note that if $L/K$ is a ramified
$C_p$-extension then $L/K$ has a single upper and lower
ramification break $u_1=b_1$.  Thus we may refer simply
to the ramification break of $L/K$.

     The upper ramification subgroups of $G$ are defined
for real $x\ge0$ by $G^0=G_0=G$, $G^x=G_{b_1}$ for
$0<x\le u_1$, $G^x=G_{b_i}$ for $u_{i-1}<x\leq u_i$, and
$G^x=\{\id_L\}$ for $x>u_n$.  Thus $u\ge0$ is an upper
ramification break of $L/K$ if and only if
$G^{u+\epsilon}\lneqq G^u$ for all $\epsilon>0$.  The
following theorem shows that the groups $G^x$ and the
upper ramification breaks are compatible with passage to
quotients of Galois groups:

\begin{theorem}[Herbrand] \label{herb}
Let $L/K$ be a finite totally ramified Galois extension
and let $M/K$ be a Galois subextension of $L/K$.  Set
$G=\Gal(L/K)$ and $H=\Gal(L/M)$.
\begin{enumerate}[(a)]
\item For $x\ge0$ we have $(G/H)^x=G^xH/H$.
\item $\U_{M/K}\subset\U_{L/K}$.
\end{enumerate}
\end{theorem}

\begin{proof}
Statement~(a) is proved as Proposition~14 in
\cite[IV]{cl}.  Statement~(b) follows easily from (a).
\end{proof}

\begin{cor} \label{quotcond}
Let $x\ge0$.  Then $x\not\in\U_{M/K}$ if and only if
$G^x\le G^{x+\epsilon}H$ for all sufficiently small
$\epsilon>0$.
\end{cor}

\begin{proof}
This follows from (a) since $G^xH=G^{x+\epsilon}H$ if
and only if $G^x\le G^{x+\epsilon}H$.
\end{proof}

     We will make frequent use of the following
(presumably well-known) fact:

\begin{lemma} \label{fact1}
Let $N/K$ be a finite totally ramified Galois extension
and set $G=\Gal(N/K)$.  Assume that $Z(G)$ contains a
subgroup $H$ such that $H\cong C_p^2$, and let $M=N^H$
be the fixed field of $H$.  Suppose there are $u<v$ such
that $u,v\not\in\U_{M/K}$ and
$\U_{N/K}=\U_{M/K}\cup\{u,v\}$.  Let $b<c$ be the lower
ramification breaks of $N/K$ that correspond to $u,v$
and let $\SS$ denote the set of fields $L$ such that
$M\subset L\subset N$ and $[L:M]=p$.  Then there is
$L_0\in\SS$ with the following properties:
\begin{enumerate}[(a)]
\item $\U_{L_0/K}=\U_{M/K}\cup\{u\}$ and $N/L_0$ has
ramification break $c$.
\item For all $L\in\SS$ such that $L\not=L_0$ we have
$\U_{L/K}=\U_{M/K}\cup\{v\}$ and $N/L$ has ramification
break $b$.
\end{enumerate}
\end{lemma}

\begin{proof}
First we prove that the lower ramification breaks of
$N/M$ are $b,c$.  Since $b$ is a lower ramification
break of $L/K$ there exists $g\in G_b\smallsetminus
G_{b+1}$.  Since $u\not\in\U_{M/K}$, it follows from
Corollary~\ref{quotcond} that $G^u\le G^{u+\epsilon}H$
for sufficiently small $\epsilon>0$.  Since $G^u=G_b$
and $G^{u+\epsilon}=G_{b+1}$ it follows from
Theorem~\ref{herb}(a) that $G_b\le G_{b+1}H$.  Hence
there are $g'\in G_{b+1}$ and $h\in H$ such that
$g=g'h$.  It follows that $h=(g')^{-1}g\in
G_b\smallsetminus G_{b+1}$, so we have $i(h)=b$.  Thus
$b$ is a lower ramification break of $N/M$.  A similar
argument shows that $c$ is a lower ramification break of
$N/M$.  Now since the lower ramification breaks of $N/M$
are $b,c$, for each $L\in \SS$ the ramification break of
$N/L$ is either $b$ or $c$. Let $L_0=N^{H_c}$ be the
fixed field of $H_c=H_{b+\epsilon}$.  Since the
ramification break of $N/L$ is $\ge c$ if and only if
$\Gal(N/L)\le H_c$ we see that $N/L_0$ has ramification
break $c$, and $N/L$ has ramification break $b$ for all
$L\in\SS$ with $L\not=L_0$.

    To complete the proof let $L\in\SS$ and set
$A=\Gal(N/L)\le H$.  Since $c$ is the largest lower
break of $N/M$, for sufficiently small $\epsilon>0$ we
have
\[G^{v+\epsilon}\cap H=G_{c+1}\cap H
=H_{c+1}=\{\id_L\}.\]
It follows that $G^{v+\epsilon}\cap A=\{\id_L\}$, so we
get $|G^{v+\epsilon}A:G^{v+\epsilon}|=|A|=p$.  Since $v$
is an upper break of $N/K$ with multiplicity 1 we have
$|G^v:G^{v+\epsilon}|=p$.  Hence $A\leq G^v$ if and only
if $G^v\le G^{v+\epsilon}A$.  By
Corollary~\ref{quotcond} we deduce that
$v\not\in\U_{L/K}$ if and only if $A\le G^v=G_c$.  Hence
$v\not\in\U_{L/K}$ if and only if $A\le G_c\cap H=H_c$.
Since $\U_{L/K}$ is equal to either $\U_{M/K}\cup\{u\}$
or $\U_{M/K}\cup\{v\}$, we conclude that
$\U_{L/K}=\U_{M/K}\cup\{u\}$ if $L=L_0$ and
$\U_{L/K}=\U_{M/K}\cup\{v\}$ if $L\not=L_0$.
\end{proof}

\section{Embedding problems in characteristic $p$}
\label{emb}

Let $K$ be a field, let $L/K$ be a finite Galois
extension, and set $G=\Gal(L/K)$.  Let $\Gt$ be a finite
group and let $\phi:\Gt\ra G$ be an onto homomorphism.
A solution to the embedding problem associated to
$(L/K,\Gt,\phi)$ is a finite extension $M/L$ such that
$M$ is Galois over $K$ and there is an isomorphism of
exact sequences
\begin{equation*}
\setlength{\arraycolsep}{1pt}
\begin{array}{*{9}c}
1&\Lrightarrow&\Gal(M/L)&\Lrightarrow&\Gal(M/K)
&\Lrightarrow&\Gal(L/K)&\Lrightarrow&1 \\[1mm]
&&\Ldownarrow&&\Ldownarrow&&\parallel&& \\[-1mm]
1&\Lrightarrow&\ker\phi&\Lrightarrow&\Gt
&\overset{\tst\phi}{\Lrightarrow}&G&\Lrightarrow&1.
\end{array}
\end{equation*}
In this section we use a theorem of Witt to show that
certain embedding problems involving local fields of
characteristic $p$ always admit a solution which is a
totally ramified extension.

     Recall that the rank of a finite $p$-group $G$ is
the minimum size of a generating set for $G$.  Let
$\Phi(G)$ denote the Frattini subgroup of $G$.  It
follows from the Burnside basis theorem that the
Frattini quotient $G/\Phi(G)$ is an elementary abelian
$p$-group such that $\rank(G)$ is equal to
$\rank(G/\Phi(G))$.  In \cite[III]{witt}, Witt proved
the following:

\begin{theorem} \label{fields}
Let $K$ be a field of characteristic $p$ and let $L/K$
be a finite Galois extension such that $G=\Gal(L/K)$ is
a $p$-group.  Let $\Gt$ be a finite $p$-group such that
$\rank(\Gt)=\rank(G)$ and let $\phi:\Gt\ra G$ be an onto
homomorphism.  Then there is an extension $M/L$ which
solves the embedding problem associated to
$(L/K,\Gt,\phi)$.
\end{theorem}

     We will use the following application of Witt's
theorem:

\begin{cor} \label{embedram}
Let $K$ be a local field of characteristic $p$ and let
$L/K$ be a finite totally ramified Galois extension
whose Galois group $G=\Gal(L/K)$ is a $p$-group.  Let
$\Gt$ be a finite $p$-group and let $\phi:\Gt\ra G$ be
an onto group homomorphism.  Then there is a totally
ramified field extension $M/L$ which solves the
embedding problem associated to $(L/K,\Gt,\phi)$.
\end{cor}

\begin{proof}
Let $N=\ker\phi$.  It suffices to consider the case
where $N\cong C_p$.  If the extension $\Gt$ of $G$ by
$N$ is split then $\Gt\cong C_p\times G$.  In this case
choose a ramified $C_p$-extension $F/K$ whose
ramification break is greater than all the upper breaks
of $L/K$.  Then $L$ and $F$ are linearly disjoint over
$K$, so $M=LF$ is a totally ramified Galois extension of
$K$ with $\Gal(M/K)\cong C_p\times G\cong\Gt$.  Hence
$M$ solves the given embedding problem.  If the
extension $\Gt$ of $G$ by $N$ is not split we claim that
$\rank(\Gt)=\rank(G)$.  We clearly have
$\rank(\Gt)\ge\rank(G)$.  Let $A$ be a generating set
for $G$ such that $|A|=\rank(G)$ and let $\At\subset\Gt$
satisfy $|\At|=|A|$ and $\phi(\At)=A$.  If
$\langle\At\rangle\not=\Gt$ then
$\langle\At\rangle\cong G$ and
$\Gt=N\times\langle\At\rangle$.  This contradicts the
assumption that our extension is not split, so we must
have $\langle\At\rangle=\Gt$.  Hence
$\rank(\Gt)=\rank(G)$.  It follows by
Theorem~\ref{fields} that there is a field extension
$M/L$ which solves the given embedding problem.  If
$M/L$ is unramified let $F$ denote the unramified
extension of $K$ of degree $p$.  Then $F\subset M$, so
we get $M=LF$ and $\Gal(M/K)\cong C_p\times G$.  This is
a contradiction, so $M/L$ is a totally ramified
extension.
\end{proof}

\section{Minimal nonabelian $p$-groups} \label{min}

We put a partial order on finite $p$-groups by
$H\preccurlyeq G$ if $H$ is isomorphic to a quotient of
$G$.  We are interested in the groups which are
$\preccurlyeq$-minimal among nonabelian $p$-groups.  We
call such a group a minimal nonabelian $p$-group.

\begin{prop} \label{abcd}
Let $p>2$ and let $G$ be a $p$-group.  Then $G$ is a
minimal nonabelian $p$-group if and only if $G$
satisfies the following conditions:
\begin{enumerate}[(i)]
\item $G$ is nilpotent of class 2.
\item $Z(G)$ is cyclic of order $p^d$ for some $d\ge1$.
\item $[G,G]$ is the subgroup of $Z(G)$ of order $p$.
\item $\Gbar:=G/Z(G)$ is an elementary abelian $p$-group
of rank $2n$ for some $n\ge1$, and $[\;,\:]$ induces a
nondegenerate skew-symmetric $\F_p$-bilinear form
$(\;,\:)_{\Gbar}$ on $\Gbar$ with values in $[G,G]$.
\end{enumerate}
\end{prop}

\begin{proof}
Suppose $G$ is a minimal nonabelian $p$-group.  Since
$Z(G)$ is nontrivial, $\Gbar$ is abelian by the
minimality of $G$.  Hence $G$ is nilpotent of class 2,
which gives (i).  Let $N$ be a nontrivial normal
subgroup of $G$.  Then $G/N$ is abelian by the
minimality of $G$, so $[G,G]\le N$.  Hence $[G,G]$ is
contained in all nontrivial subgroups of $Z(G)$, so
$Z(G)$ is cyclic and $[G,G]$ is the unique subgroup of
$Z(G)$ of order $p$.  This proves (ii) and (iii).

     Let $z$ be a generator for $Z(G)\cong C_{p^d}$ and
set $w=z^{p^{d-1}}$.  Then $[G,G]=\langle w\rangle$.
For $x,y\in G$ we have $xyx^{-1}=yw^a$ for some
$a\in\Z$.  Hence $xy^px^{-1}=y^pw^{pa}=y^p$, so
$y^p\in Z(G)$.  It follows that $\Gbar$ is an
elementary abelian $p$-group.  Let $x,y_1,y_2\in G$.
Then there are $a_j\in\Z$ such that
$xy_jx^{-1}=y_jw^{a_j}$ for $j=1,2$.  It follows that
$xy_1y_2x^{-1}=y_1y_2w^{a_1+a_2}$, and hence that
$[x,y_1y_2]=[x,y_1][x,y_2]$.  Since $[y,x]=[x,y]^{-1}$,
we deduce that $[\;,\:]$ induces a skew-symmetric
$\F_p$-bilinear pairing $(\;,\:)_{\Gbar}$ on $\Gbar$.  If
$[x,y]=1$ for all $y\in G$ then $x\in Z(G)$, so the
pairing is nondegenerate.  Therefore $\Gbar$ has even
$\F_p$-rank.  This proves (iv).

     Conversely, suppose (i)--(iv) hold.  Then $G$ is
nonabelian by (i) or (iii).  Let $N$ be a nontrivial
normal subgroup of $G$.  Then $N\cap Z(G)$ is
nontrivial, so $[G,G]\le N$ by (ii) and (iii).  Hence
$G/N$ is abelian, so $G$ is a minimal nonabelian
$p$-group.
\end{proof}

     The minimal nonabelian $p$-groups can be described
more explicitly.  For $n,d\ge1$ we define a group
$H(n,d)$ of order $p^{2n+d}$ generated by
$x_1,\ldots,x_n,y_1,\ldots,y_n,z$, with $|x_i|=|y_i|=p$
and $|z|=p^d$.  All these generators commute with each
other, except for $x_i$ and $y_i$, which satisfy
$[x_i,y_i]=z^{p^{d-1}}$ for $1\le i\le n$.  Thus
$H(1,1)$ is the Heisenberg $p$-group, and $H(n,1)$ is an
extraspecial $p$-group.

     We define another group $A(n,d)$ of order
$p^{2n+d}$ generated by
$x_1,\ldots,x_n,y_1,\ldots,y_n,z$.  In $A(n,d)$ we have
$|x_i|=p$ for $2\le i\le n$, $|y_i|=p$ for
$1\le i\le n$, and $x_1^p=z$ with $|z|=p^d$.  As with
$H(n,d)$, all generators commute with each other except
for $x_i$ and $y_i$, which satisfy
$[x_i,y_i]=z^{p^{d-1}}$ for $1\le i\le n$.  Thus
$A(1,1)$ is the metacyclic group of order $p^3$, and
$A(n,1)$ is an extraspecial $p$-group.

     It is clear from the constructions that the groups
$H(n,d)$ and $A(n,d)$ satisfy conditions (i)--(iv) of
Proposition~\ref{abcd}.  We now prove the converse,
which states that every minimal nonabelian $p$-group is
isomorphic to one of these groups.

\begin{prop} \label{classify}
Let $p>2$ and let $G$ be a minimal nonabelian $p$-group.
Then either $G\cong H(n,d)$ or $G\cong A(n,d)$ for some
$n,d\ge1$.
\end{prop}

\begin{proof}
Since $G$ is a minimal nonabelian $p$-group, $G$
satisfies conditions (i)--(iv) of Proposition~\ref{abcd}.
Let $z$ be a generator for $Z(G)$; then $|z|=p^d$ for
some $d\ge1$.  Set $w=z^{p^{d-1}}$, so that
$\langle w\rangle=[G,G]$.  It follows from (iv) that
$x^p\in Z(G)$ for all $x\in G$.  For $x\in G$
set $\xbar=xZ(G)\in \Gbar$.

     Suppose that $x^p\in Z(G)^p$ for all $x\in G$.
Since $[\;,\:]$ induces a nondegenerate $\F_p$-linear
pairing $(\;,\:)_{\Gbar}$ on $\Gbar$, there is an
$\F_p$-basis
$\{\xbar_1,\ldots,\xbar_n,\ybar_1,\ldots,\ybar_n\}$ for
$\Gbar$ such that
$(\xbar_i,\xbar_j)_{\Gbar}=(\ybar_i,\ybar_j)_{\Gbar}=w^0$ and
$(\xbar_i,\ybar_j)_{\Gbar}=w^{\delta_{ij}}$ for all
$1\le i,j\le n$.  Let $x_i'\in G$ be such that
$\xbar_i=x_i'Z(G)$.  Then there is $a_i\in\Z$ such that
$(x_i')^p=z^{pa_i}$.  Therefore $x_i=x_i'z^{-a_i}$
satisfies $x_iZ(G)=\xbar_i$ and $|x_i|=p$.  Similarly,
there are $y_i\in G$ with $y_iZ(G)=\ybar_i$ and
$|y_i|=p$.  It follows that $G\cong H(n,d)$.

     Now assume that there exists $x\in G$ such that
$x^p\not\in Z(G)^p$.  Define $\phi:\Gbar\ra Z(G)/Z(G)^p$
by $\phi(xZ(G))=x^pZ(G)^p$.  Then $\phi$ is clearly
well-defined.  We claim that $\phi$ is a group
homomorphism, and hence an $\F_p$-linear map.  Let
$x,y\in G$; then $[x,y]=z^a$ for some integer $a$.  Thus
$yx=xyz^{-a}$, so we get $(xy)^p=x^py^pz^{pb}$ with
$b=-\frac12(p-1)a$.  Hence $\phi(xy)=\phi(x)\phi(y)$.
By our assumption, $\phi$ is nontrivial, so
$\phi(\Gbar)=Z(G)/Z(G)^p$ is cyclic of order $p$.  Set
$V=\ker\phi$ and let $V^{\perp}$ be the orthogonal
complement of $V$ with respect to the pairing
$(\;,\:)_{\Gbar}$.  Then $V$ and $V^{\perp}$ are
$\F_p$-subspaces of $\Gbar$, with $\dim_{\F_p}(V)=2n-1$
and $\dim_{\F_p}(V^{\perp})=1$.  Let $y_1'\in G$ be such
that $\ybar_1'=y_1'Z(G)$ generates $V^{\perp}$.  Since
$[y_1',y_1']=w^0$ we have
$\ybar_1'\in(V^{\perp})^{\perp}=V=\ker\phi$.  Hence
there is $a\in\Z$ such that $(y_1')^p=z^{ap}$.  Then
$y_1=y_1'z^{-a}$ satisfies $\ybar_1=\ybar_1'$ and
$|y_1|=p$.  Now let $x_1\in G$ be such that
$(\xbar_1,\ybar_1)_{\Gbar}=w^1$.  Then $\xbar_1\not\in V$, so
$x_1^p$ is a generator for $Z(G)$.  Therefore we may
assume that $x_1^p=z$.  Let $W$ denote the span of
$\{\xbar_1,\ybar_1\}$ in $\Gbar$.  Then the restriction
of $(\;,\:)_{\Gbar}$ to $W$ is nondegenerate, so
$V=W\oplus W^{\perp}$ and the restriction of
$(\;,\:)_{\Gbar}$
to $W^{\perp}$ is a nondegenerate skew-symmetric
$\F_p$-bilinear form.  Hence there is a basis
$\{\xbar_2,\ldots,\xbar_n,\ybar_2,\ldots,\ybar_n\}$ for
$W^{\perp}$ such that
$(\xbar_i,\xbar_j)_{\Gbar}=(\ybar_i,\ybar_j)_{\Gbar}=w^0$ and
$(\xbar_i,\ybar_j)_{\Gbar}=w^{\delta_{ij}}$ for $2\le i,j\le n$.
Let $x_i'\in G$ be such that $\xbar_i=x_i'Z(G)$, and let
$a_i\in\Z$ satisfy $(x_i')^p=z^{a_i}$.  Since
$2\le i\le n$ we have $a_i=pb_i$ for some $b_i\in\Z$.
Hence $x_i=x_i'z^{-b_i}$ satisfies $x_iZ(G)=\xbar_i$ and
$|x_i|=p$.  Similarly for $2\le i\le n$ there are
$y_i\in G$ such that $y_iZ(G)=\ybar_i$ and $|y_i|=p$.
Therefore $G\cong A(n,d)$.
\end{proof}

\begin{remark}
A $p$-group $G$ is said to be of symplectic type if
every abelian characteristic subgroup of $G$ is cyclic.
Philip Hall, in unpublished notes, showed that for
$p>2$ the nonabelian $p$-groups of symplectic type are
precisely the groups $H(n,d)$ and $A(n,d)$ for $n,d\ge1$.
Therefore for $p>2$ the minimal nonabelian $p$-groups
are the same as the nonabelian $p$-groups of symplectic
type.  A proof of Hall's result can be found in
\cite[5.4.9]{Go80}.  We thank Peter Sin for pointing us
to this reference.
\end{remark}

     Let $G$ be a group and let $N_1,N_2$ be subgroups
of $G$.  Say that $G$ is a central product of $N_1$ and
$N_2$ if $N_1\cup N_2$ generates $G$ and every element
of $N_1$ commutes with every element of $N_2$.  In that
case there is a subgroup $A$ of $Z(N_1)\times Z(N_2)$
such that $G\cong(N_1\times N_2)/A$.

     We wish to express minimal nonabelian $p$-groups as
central products, with $H(1,1)$ as one of the factors.
For convenience we extend the definition of $H(n,d)$ by
setting $H(0,d)=C_{p^d}$.

\begin{prop} \label{central}
\begin{enumerate}[(a)]
\item Let $n,d\ge1$.  Then $H(n,d)$ is a central product
of subgroups $N_1$ and $N_2$, with $N_1\cong H(n-1,d)$
and $N_2\cong H(1,1)$.  More precisely,
\begin{equation} \label{cenH}
H(n,d)\cong(H(n-1,d)\times H(1,1))/B
\end{equation}
for some subgroup $B$ of $Z(H(n-1,d))\times Z(H(1,1))$
of order $p$.
\item Let $n\ge2$ and $d\ge1$.  Then $A(n,d)$ is a
central product of subgroups $N_1$ and $N_2$, with
$N_1\cong A(n-1,d)$ and $N_2\cong H(1,1)$.  More
precisely,
\begin{equation} \label{cenA}
A(n,d)\cong(A(n-1,d)\times H(1,1))/B
\end{equation}
for some subgroup $B$ of $Z(H(n-1,d))\times Z(H(1,1))$
of order $p$.
\end{enumerate}
\end{prop}

\begin{proof}
(a) Let $N_1$ be the subgroup of $H(n,d)$ generated by
$x_1,\ldots,x_{n-1},y_1,\ldots,y_{n-1},z$ and let $N_2$
be the subgroup of $H(n,d)$ generated by
$x_n,y_n,z^{p^{d-1}}$.  Then $N_1\cong H(n-1,d)$,
$N_2\cong H(1,1)$, and $N_1,N_2$ satisfy the conditions
for a central product.  Therefore there is $B$
satisfying (\ref{cenH}).  Since
$|H(n,d)|=p^{2n+d}$, $|H(n-1,d)|=p^{2n+d-2}$, and
$|H(1,1)|=p^3$, we must have $|B|=p$.
\\[\smallskipamount]
(b) Let $N_1$ be the subgroup of $A(n,d)$ generated by
$x_1,\ldots,x_{n-1},y_1,\ldots,y_{n-1},z$ and let $N_2$
be the subgroup of $A(n,d)$ generated by
$x_n,y_n,z^{p^{d-1}}$.  Then $N_1\cong A(n-1,d)$,
$N_2\cong H(1,1)$, and $N_1,N_2$ satisfy the conditions
for a central product.  Therefore there is $B$ satisfying
(\ref{cenA}).  Since $|A(n,d)|=p^{2n+d}$,
$|A(n-1,d)|=p^{2n+d-2}$, and $|H(1,1)|=p^3$, we must
have $|B|=p$.
\end{proof}

Proposition~\ref{central}(b) does not apply to
groups of the form $A(1,d)$.  Instead, we use the
following description:

\begin{prop} \label{A1d}
Let $d\ge1$, and write $H(1,1)=\langle x_1,y_1,z\rangle$,
$C_{p^{d+1}}=\langle w\rangle$.  Define a subgroup $G_d$
of $H(1,1)\times C_{p^{d+1}}$ by
$G_d=\langle x_1w,y_1,z\rangle$, and set
\[\Gbar_d=G_d/\langle(x_1w)^{p^d}z^{-1}\rangle
=G_d/\langle w^{p^d}z^{-1}\rangle.\] 
Then $\Gbar_d\cong A(1,d)$.
\end{prop}

\begin{proof}
We have $[x_1w,y_1]=[x_1,y_1]=z$ and $(x_1w)^p=w^p$.
Let $\xbar_1,\ybar_1,\zbar$ denote the images in
$\Gbar_d$ of $x_1w,y_1,z$.  Then $|\xbar_1|=p^{d+1}$,
$|\ybar_1|=p$, and
$[\xbar_1,\ybar_1]=\zbar=\xbar_1^{p^d}$.  Hence
$\Gbar_d\cong A(1,d)$.
\end{proof}

\section{$p$-extensions with a nonintegral upper break}
\label{pext}

Let $p>2$, let $G$ be a nonabelian $p$-group, and let
$K$ be a local field of characteristic $p$ with perfect
residue field $k$.  In this section we prove that there
exists a totally ramified Galois extension $L/K$ with
Galois group $G$ such that $L/K$ has an upper
ramification break which is not an integer.  It follows
from Theorem~\ref{fields} that every embedding problem
over $K$ which only involves $p$-groups can be solved
with a totally ramified extension.  Therefore we only
need to give an example of a $G$-extension with a
nonintegral upper break for each $G$ which is
$\preccurlyeq$-minimal among nonabelian $p$-groups.
These groups are classified in
Proposition~\ref{classify}.

     Our proof uses a bootstrap argument, based on
constructing $H(1,1)$-extensions with a nonintegral
upper ramification break.  As a first step, we give an
easy method for building $H(1,1)$-extensions using
Artin-Schreier extensions.  Recall that if $\beta\in K$
satisfies $v_K(\beta)=-b$ with $b\ge1$ and $p\nmid b$
then the roots of $X^p-X-\beta$ generate a
$C_p$-extension of $K$ with ramification break $b$ (see
Proposition~2.5 in \cite[III]{FV}).

\begin{lemma} \label{H11}
Let $a,b$ be positive integers with $a>b$, $p\nmid a$,
and $p\nmid b$.  Let $\alpha,\beta\in K$ satisfy
$v_K(\alpha)=-a$ and $v_K(\beta)=-b$ and let
$x,y\in K^{sep}$ satisfy $x^p-x=\alpha$ and
$y^p-y=\beta$.  Set $M=K(x,y)$ and let $\gamma\in K$.
Let $z\in K^{sep}$ satisfy $z^p-z=\alpha y+\gamma$ and
set $L=M(z)$.  Then $L/K$ is a totally ramified
$H(1,1)$-extension.
\end{lemma}

\begin{proof}
By construction $M/K$ is a totally ramified
$C_p^2$-extension.  Let $\sigma,\tau\in\Gal(M/K)$
satisfy $\sigma(x)=x+1$, $\sigma(y)=y$, $\tau(x)=x$, and
$\tau(y)=y+1$.  Then
\begin{align*}
(\sigma-1)(\alpha y+\gamma)&=0 \\
(\tau-1)(\alpha y+\gamma)&=\alpha=\wp(x).
\end{align*}
Since $x\in M$ it follows that $L/K$ is Galois.
Furthermore, we may extend $\sigma,\tau$ to
$\sigmat,\taut\in\Gal(L/K)$ by setting $\sigmat(z)=z$
and $\taut(z)=z+x$.  We easily find that
$|\sigmat|=|\taut|=p$, $[\sigmat,\taut]\in\Gal(L/M)$,
and $[\sigmat,\taut](z)=z+1$.  The last formula implies
that $[\sigmat,\taut]$ generates $\Gal(L/M)$.  Therefore
$\Gal(L/K)\cong H(1,1)$.
\end{proof}

\begin{prop} \label{p3}
Let $K$ be local field of characteristic $p>2$ and let
$F/K$ be a ramified $C_p$-extension.  Let $b$ be the
ramification break of $F/K$, and let $a$ be an integer
such that $a>b$ and $a\not\equiv0,-b\pmod{p}$.  Then
there is a totally ramified extension $L/F$ such that
$L/K$ is an $H(1,1)$-extension with
$\U_{L/K}=\{b,a,a+p^{-1}b\}$.  In particular, $L/K$ has
an upper ramification break which is not an integer.
\end{prop}

\begin{proof}
It follows from Proposition~2.4 in \cite[III]{FV} that
there is $y\in F$ such that $F=K(y)$, $v_F(y)=-b$, and
$\beta:=y^p-y\in K$.  Since $p\nmid b$ we can write
$a=bt+ps$ with $0\le t<p$; by our assumptions on $a$ we
get $1\le t\le p-2$.  Set $\alpha=\pi_K^{-ps}\beta^t$;
then $v_K(\alpha)=-a$.  Let $x\in K^{sep}$ satisfy
$x^p-x=\alpha$.  Then $M:=F(x)=K(x,y)$ is a
$C_p^2$-extension of $K$ with upper ramification breaks
$b,a$.  Let $r$ be the inverse of $t+1$ in
$\F_p^{\times}$, let $z\in K^{sep}$ satisfy
$z^p-z=\alpha y+r\alpha\beta$, and set $L=M(z)$.  Then
$L/K$ is an $H(1,1)$-extension by Lemma~\ref{H11}.
Furthermore, $\Gal(L/M)\cong C_p$ is the commutator
subgroup of $\Gal(L/K)$.  Hence by (\ref{bigger}),
$\Gal(L/M)$ is the smallest nontrivial ramification
subgroup of $\Gal(L/K)$.  Therefore by the corollary to
Proposition~3 in \cite[IV]{cl}, the lower ramification
breaks of $M/K$ are also lower ramification breaks of
$L/K$.

     Let $E=F(z)$.  We can't directly compute the
ramification break of the $C_p$-extension $E/F$, since
$v_F(\alpha y+r\alpha\beta)=v_F(r\alpha\beta)$ is
divisible by $p$.  So instead we consider the
Artin-Schreier equation
\[X^p-X=\alpha y+r\alpha\beta-\wp(r\pi_K^{-s}y^{t+1}).\]  
Since $r\pi_K^{-s}y^{t+1}\in F$, the roots of this
equation generate $E$ over $F$.  Furthermore, we have
\begin{align*}
\alpha y+r\alpha\beta-\wp(r\pi_K^{-s}y^{t+1})
&=\alpha y+r\alpha\beta
+r\pi_K^{-s}y^{t+1}-(r\pi_K^{-s}y^{t+1})^p \\
&=\alpha y+r\alpha\beta
+r\pi_K^{-s}y^{t+1}-r\pi_K^{-ps}(y+\beta)^{t+1} \\
&=\alpha y+r\alpha\beta+r\pi_K^{-s}y^{t+1}
-r\pi_K^{-ps}\sum_{i=0}^{t+1}\binom{t+1}{i}\beta^{t+1-i}y^i.
\end{align*}
Since $\alpha=\pi_K^{-ps}\beta^t$, the $i=0$ term in the
sum is $-r\alpha\beta$ and the $i=1$ term is
$-\alpha y$.  It follows that
\[\alpha y+r\alpha\beta-\wp(r\pi_K^{-s}y^{t+1})
=r\pi_K^{-s}y^{t+1}-r\pi_K^{-ps}\sum_{i=2}^{t+1}
\binom{t+1}{i}\beta^{t+1-i}y^i.\]
Since $1\le t\le p-2$ we get
\begin{align*}
v_F\left(r\pi_K^{-ps}\binom{t+1}{2}\beta^{t-1}y^2\right)
&=-p^2s-(t-1)pb-2b \\
&=-pa+pb-2b.
\end{align*}
Since $a>b$ we have
\[v_F(r\pi_K^{-s}y^{t+1})=-ps-(t+1)b=-a-b>-pa+pb-2b.\]
Therefore
\begin{align*}
v_F(\alpha y+r\alpha\beta-\wp(r\pi_K^{-s}y^{t+1}))
&=v_F\left(r\pi_K^{-ps}\binom{t+1}{2}\beta^{t-1}y^2\right) \\
&=-pa+pb-2b,
\end{align*}
which is not divisible by $p$.  Hence the ramification
break of $E/F$ is $2b+p(a-b)$.

     Since $b$ is the ramification break of $F/K$, it
follows from Lemma~\ref{fact1} that the ramification
break of $M/F$ is $b+p(a-b)$.  Hence the upper breaks of
the $C_p^2$-extension $L/F$ are $b+p(a-b),2b+p(a-b)$,
and the lower breaks are $b+p(a-b),b+pa$.  Since the
upper breaks of the $C_p^2$-extension $M/K$ are $b,a$,
with $b<a$, the lower breaks of this extension are
$b,b+p(a-b)$.  As we observed above, these are also
lower breaks of $L/K$.  Hence the lower breaks of $L/K$
are $b,b+p(a-b),b+pa$.  We conclude that the upper
ramification breaks of $L/K$ are $b,a$, and
\[a+p^{-2}((b+pa)-(b+p(a-b)))=a+p^{-1}b.\qedhere\]
\end{proof}

\begin{lemma} \label{fact2}
Let $N/K$ be a totally ramified Galois extension such
that $\Gal(N/K)$ is isomorphic to either $H(n,d)$ (with
$n\ge0$ and $d\ge1$) or $A(n,d)$ (with $n,d\ge1$).  Let
$G=\Gal(N/K)$, let $H$ be the unique subgroup of $Z(G)$
of order $p$, and let $M\subset N$ be the fixed field of
$H$.  Then $\U_{N/K}=\U_{M/K}\cup\{v\}$ for some
$v\in\Q$ such that $v>w$ for all $w\in\U_{M/K}$.
Furthermore, $H=G^v$ is the smallest nontrivial
ramification subgroup of $G$.
\end{lemma}

\begin{proof}
Let $\sigma\in G$ with $\sigma\not\in Z(G)$.  Then there
is $\tau\in G$ such that $[\sigma,\tau]$ generates
$H\cong C_p$.  Hence by (\ref{bigger}), for $\rho\in H$
we have $i(\rho)>i(\sigma)$.  Suppose $\sigma\in Z(G)$
but $\sigma\not\in H$.  Then there is $1\le i\le d-1$
such that $\sigma^{p^i}$ generates $H$.  Once again
by (\ref{bigger}) we get $i(\rho)>i(\sigma)$ for all
$\rho\in H$.  It follows that $H$ is the smallest
nontrivial ramification subgroup of $G$.  Therefore
$H=G^v$, with $v$ the largest upper ramification break
of $N/K$.  Using Theorem~\ref{herb}(b) we deduce that
$\U_{N/K}=\U_{M/K}\cup\{v\}$.
\end{proof}

     We now construct $H(n,d)$-extensions and
$A(n,d)$-extensions which have at least one nonintegral
upper break.

\begin{prop} \label{nonint}
Let $K$ be a local field of characteristic $p>2$ and let
$n,d\ge1$.
\begin{enumerate}[(a)]
\item There is a totally ramified Galois extension
$L/K$ such that $\Gal(L/K)\cong H(n,d)$ and the largest
upper ramification break of $L/K$ is not an integer.
\item There is a totally ramified Galois extension
$L/K$ such that $\Gal(L/K)\cong A(n,d)$ and the largest
upper ramification break of $L/K$ is not an integer.
\end{enumerate}
\end{prop}

\begin{proof}
(a) By Corollary~\ref{embedram} there exists a totally
ramified $H(n-1,d)$-extension $N_1/K$.  Let $v$ be the
largest upper ramification break of $N_1/K$ and let
$a,b$ be integers such that $a>b>v$, $p\nmid b$, and
$a\not\equiv0,-b\pmod{p}$.  Then by Proposition~\ref{p3}
there is an $H(1,1)$-extension $N_2/K$ such that
$\U_{N_2/K}=\{b,a,a+p^{-1}b\}$.  Set $N=N_1N_2$.  Since
$\U_{N_1/K}$ and $\U_{N_2/K}$ are disjoint we have
$\U_{N/K}=\U_{N_1/K}\cup\U_{N_2/K}$, and
$N_1\cap N_2=K$.  It follows that
\begin{align*}
\Gal(N/K)&\cong\Gal(N_1/K)\times\Gal(N_2/K) \\
&\cong H(n-1,d)\times H(1,1).
\end{align*}
For $i=1,2$ let $M_i$ be the subfield of $N_i$ fixed by
the unique subgroup of $Z(\Gal(N_i/K))$ with order $p$.
Set $M=M_1M_2$; then $\Gal(N/M)\cong C_p^2$.  It follows
from Lemma~\ref{fact2} that
$\U_{N_1/K}=\U_{M_1/K}\cup\{v\}$ and
$\U_{N_2/K}=\U_{M_2/K}\cup\{a+p^{-1}b\}$.  Since
$\U_{M/K}=\U_{M_1/K}\cup\U_{M_2/K}$ this implies
$\U_{N/K}=\U_{M/K}\cup\{v,a+p^{-1}b\}$.  By
Proposition~\ref{central}(a) there is a subgroup
$B\le\Gal(N/M)$ such that $\Gal(N/M)/B\cong H(n,d)$.
Let $L=N^B$ be the fixed field of $B$; then $L/K$ is a
totally ramified $H(n,d)$-extension.  Since
\[\Gal(N_1M_2/K)\cong H(n-1,d)\times C_p^2
\not\cong H(n,d)\]
we have $L\not=N_1M_2$.  Since $a+p^{-1}b>v$ and $v$ is
the largest upper ramification break of $N_1M_2/K$ it
follows from Lemma~\ref{fact1} that
$a+p^{-1}b\not\in\Z$ is an
upper ramification break of $L/K$.
\\[\smallskipamount]
(b) If $n\ge2$ then we proceed as in case (a): Let
$N_1/K$ be a totally ramified $A(n-1,d)$-extension whose
largest upper ramification break is $v$.  By
Proposition~\ref{p3} there is an $H(1,1)$-extension
$N_2/K$ such that $\U_{N_2/K}=\{b,a,a+p^{-1}b\}$, with
$a>b>v$, $p\nmid b$, and $a\not\equiv0,-b\pmod{p}$.
Setting $N=N_1N_2$, we get
$\U_{N/K}=\U_{N_1/K}\cup\U_{N_2/K}$, $N_1\cap N_2=K$,
and
\begin{align*}
\Gal(N/K)&\cong\Gal(N_1/K)\times\Gal(N_2/K) \\
&\cong A(n-1,d)\times H(1,1).
\end{align*}
Defining $M_i$ and $M$ as in the proof of (a) we get
$\Gal(N/M)\cong C_p^2$ and
$\U_{N/K}=\U_{M/K}\cup\{v,a+p^{-1}b\}$.  Hence by
Proposition~\ref{central}(b) there is $B\le\Gal(N/M)$
such that $\Gal(N/K)/B\cong A(n,d)$.  Setting $L=N^B$ we
get $\Gal(L/K)\cong A(n,d)$.  Since $N_1M_2$ is a
$C_p$-extension of $M$ such that $v$ is an upper
ramification break of $N_1M_2/K$, using Lemma~\ref{fact1}
we deduce that $a+p^{-1}b$ is an upper ramification break
of $L/K$.

     It remains to construct an $A(1,d)$-extension with
a nonintegral upper ramification break for each $d\ge1$.
Let $N_1/K$ be a totally ramified
$C_{p^{d+1}}$-extension and let $v$ denote the largest
upper ramification break of $N_1/K$.  Let $F/K$ be the
$C_p$-subextension of $N_1/K$.  Let $b$ be the
ramification break of $F/K$ and let $a$ be an integer
such that $a>v$ and $a\not\equiv0,-b\pmod{p}$.  Then by
Proposition~\ref{p3} there is a totally ramified
$H(1,1)$-extension $N_2/K$ such that $N_1\cap N_2=F$ and
$\U_{N_2/K}=\{b,a,a+p^{-1}b\}$.  Set $N=N_1N_2$.  Then
\[\Gal(N/K)\cong\{(\sigma_1,\sigma_2)\in
\Gal(N_1/K)\times\Gal(N_2/K):
\sigma_1|_F=\sigma_2|_F\}\]
is isomorphic to the group $G_d$ defined in
Proposition~\ref{A1d}.  Let $M_1/K$ be the
$C_{p^d}$-subextension of $N_1/K$ and let $M_2/K$ be the
$C_p^2$-subextension of $N_2/K$.  Then $M_1\cap M_2=F$.
Set $M=M_1M_2$; then $N/M$ is a $C_p^2$-extension.  By
Proposition~\ref{A1d} there is a $C_p$-subextension
$L/M$ of $N/M$ such that $L/K$ is a totally ramified
$A(1,d)$-extension.  On the other hand, $N_1M_2/M$ is a
$C_p$-subextension of $N/M$ such that $N_1M_2/K$ has $v$
as an upper ramification break.  Since
\[\Gal(N_1M_2/K)\cong C_{p^{d+1}}\times C_p
\not\cong A(1,d)\]
we have $L\not=N_1M_2$.  Hence by Lemma~\ref{fact1} we
see that $a+p^{-1}b$ is an upper ramification break of
$L/K$.  This completes the proof.
\end{proof}

     We now prove the converse of the Hasse-Arf theorem
for totally ramified $p$-extensions.

\begin{theorem} \label{pchat}
Let $K$ be a local field of characteristic $p>2$ and let
$G$ be a finite nonabelian $p$-group.  Then there is a
totally ramified $G$-extension $L/K$ which has an upper
ramification break which is not an integer.
\end{theorem}

\begin{proof}
By Proposition~\ref{classify} there is a quotient
$\Gbar=G/H$ of $G$ which is isomorphic to either
$H(n,d)$ or $A(n,d)$ for some $n,d\ge1$.  By
Proposition~\ref{nonint} there is a totally ramified
$\Gbar$-extension $M/K$ which has a nonintegral upper
ramification break.  By Corollary~\ref{embedram} there
is an extension $L/M$ such that $L/K$ is a totally
ramified $G$-extension.  Since $\U_{M/K}\subset\U_{L/K}$
it follows that $L/K$ has a nonintegral upper
ramification break.
\end{proof}

\begin{cor} \label{pchat0}
Let $p>2$ and let $G$ be a finite nonabelian $p$-group.
Then there is a local field $F$ of characteristic 0 with
residue characteristic $p$ and a totally ramified
$G$-extension $E/F$ which has an upper ramification
break which is not an integer.
\end{cor}

\begin{proof}
Let $L/K$ be a totally ramified $G$-extension with a
nonintegral upper ramification break and let $u_{L/K}$
be the largest upper ramification break of $L/K$.  Let
$F$ be a local field of characteristic 0 with the same
residue field as $K$.  Assume further that the absolute
ramification index $e_F=v_F(p)$ of $F$ satisfies
$e_F>u_{L/K}$.  Then it follows from
Th\'{e}or\`{e}me~2.8 of \cite{Del} that there is a
totally ramified $G$-extension $E/F$ with the same
ramification breaks as $L/K$.
\end{proof}

\section{$G$-extensions with a nonintegral upper break}
\label{Gext}

Let $G$ be a nonabelian group which is the Galois group
of some totally ramified extension of local fields with
residue characteristic $p>2$.  In this section we prove
the converse to the Hasse-Arf theorem for totally
ramified extensions by showing that there exists a
totally ramified $G$-extension of local fields of
characteristic $p$ which has a non-integral upper
ramification break.

     Let $K$ be a local field with residue
characteristic $p$ and let $L/K$ be a totally ramified
Galois extension of degree $mp^n$, with $p\nmid m$.  Set
$G=\Gal(L/K)$ and let $P$ be the wild ramification
subgroup of $G$.  Then $P\trianglelefteq G$ and
$G/P\cong C_m$, so we get $G\cong P\rtimes_{\psi}C_m$
for some homomorphism $\psi:C_m\ra\Aut(P)$.  In Remark~1
of \cite{fes} it is shown that if $P$ is abelian but $G$
is nonabelian then $L/K$ has a nonintegral upper
ramification break.  The following result gives a larger
class of groups $G$ such that every totally ramified
$G$-extension has a nonintegral upper ramification
break.

\begin{prop} \label{always}
Let $K$ be a local field with residue characteristic $p$
and let $L/K$ be a totally ramified Galois extension of
degree $mp^n$, with $p\nmid m$.  Set $G=\Gal(L/K)$ and
write $G\cong P\rtimes_{\psi}C_m$ as above.  If the
action of $C_m$ on $P$ is nontrivial then $L/K$ has a
nonintegral upper ramification break.
\end{prop}

\begin{proof}
Let $\Phi(P)$ be the Frattini subgroup of $P$ and let
$M=L^{\Phi(P)}$ be the fixed field of $\Phi(P)$.  Then
$\Phi(P)\trianglelefteq G$, so $M/K$ is a Galois
extension.  Set $\Pbar=P/\Phi(P)$ and let
$\psibar:C_m\ra\Aut(\Pbar)$ be the homomorphism induced
by $\psi$.  Then $\psibar$ is nontrivial by a theorem of
Burnside (see Theorem~1.4 in Chapter~5 of \cite{Go80}).
Hence
\[\Gal(M/K)\cong G/\Phi(P)\cong
\Pbar\rtimes_{\psibar}C_m\]
is nonabelian.  Since $\Pbar$ is abelian it follows from
Remark~1 in \cite{fes} that $M/K$ has a nonintegral
upper ramification break.  It then follows from
Theorem~\ref{herb}(b) that $L/K$ has a nonintegral upper
ramification break.
\end{proof}

     We now prove our converse of the Hasse-Arf theorem:

\begin{theorem} \label{chat}
Let $G$ be a finite nonabelian group which is the Galois
group of some totally ramified Galois extension of local
fields with residue characteristic $p>2$.  Then there is
a totally ramified $G$-extension $L/K$ of local fields
with residue characteristic $p$ which has a nonintegral
upper ramification break.
\end{theorem}

\begin{proof}
The assumption on $G$ implies that $G$ has a normal
Sylow $p$-subgroup $P$ such that
$G\cong P\rtimes_{\psi}C_m$ for some $m$ with $p\nmid m$
and some $\psi:C_m\ra\Aut(P)$.  If $\psi(C_m)$ is
nontrivial then it follows from Proposition~\ref{always}
that every totally ramified $G$-extension $L/K$ of local
fields with residue characteristic $p$ has a nonintegral
upper ramification break.  On the other hand, if
$\psi(C_m)$ is trivial then $G\cong P\times C_m$, so $P$
is nonabelian.  Hence by Theorem~\ref{pchat} or
Corollary~\ref{pchat0} there is a local field $F$ with
residue characteristic $p$ and a totally ramified
$P$-extension $E/F$ which has a nonintegral upper
ramification break.  Let $\zeta_m$ be a primitive $m$th
root of unity and set $K=F(\zeta_m)$, $L_1=E(\zeta_m)$.
Then $L_1/K$ is a totally ramified $P$-extension with
the same ramification breaks as $E/F$.  In addition,
there is a totally ramified $C_m$-extension $L_2/K$.
Set $L=L_1L_2$.  Then $L/K$ is a totally ramified
$G$-extension with a nonintegral upper ramification
break.
\end{proof}

\begin{remark}
Let $L/K$ be a totally ramified Galois extension of
local fields such that for every totally ramified
abelian extension $E/K$, the upper ramification breaks
of $LE/K$ are all integers.  In \cite{fes} Fesenko
proved that $\Gal(L/K)$ must be abelian in this case.
This gives a converse to the Hasse-Arf theorem of a
different sort than the one presented here.
\end{remark}

\begin{remark} \label{D4}
Let $K$ be a local field of characteristic 2 and let
$L/K$ be a totally ramified Galois extension whose
Galois group is the dihedral group $D_4$ of order 8.  It
is shown in \cite{p3} that the upper ramification breaks
of $L/K$ must be integers.  Hence the approach that we
use here to prove the converse to the Hasse-Arf theorem
by constructing extensions of local fields in
characteristic $p$ cannot be extended to the case $p=2$.
However, there are several totally ramified
$D_4$-extensions of $\Q_2$ which have nonintegral upper
ramification breaks.  For instance, the extension of
$\Q_2$ generated by a root of the polynomial
$X^8+4X^7+2X^4+4X^2+14$ is a $D_4$-extension whose upper
breaks are 1, 2, 5/2 (see \cite{lmfdb}).  As far as we
know it remains an open question whether the converse to
Hasse-Arf holds for totally ramified extensions of local
fields with residue characteristic 2.
\end{remark}

\begin{remark}
It would be interesting to know for which local fields
$K$ the following stronger converse to the Hasse-Arf
theorem holds: For every nonabelian group $G$ such that
$K$ admits a totally ramified $G$-extension, there is a
totally ramified $G$-extension $L/K$ which has a
nonintegral upper ramification break.  It follows from
the proof of Theorem~\ref{chat} that this converse to
Hasse-Arf holds for local fields of characteristic
$p>2$.  On the other hand, Remark~\ref{D4} shows that
this converse to Hasse-Arf does not hold for local
fields of characteristic 2.  As far as we know it is an
open question whether this stronger converse to
Hasse-Arf holds for local fields of characteristic 0.
\end{remark}

\end{document}